\pdfoutput=1 % FOR ARXIV
\documentclass[a4paper]{article}
\usepackage{amsthm,amssymb,amsmath,enumerate,graphicx}

\newtheorem{THM}{Theorem}[section]
\newtheorem{LEM}[THM]{Lemma}
\newtheorem{COR}[THM]{Corollary}
\newtheorem{DEF}[THM]{Definition}

\newtheorem{EX}[THM]{Example}

\def\shift(#1)(#2){\!\!\downarrow\!{}^{#1}_{\raise .1ex\vbox to 0pt{\vss\hbox{$\scriptstyle #2$}}}\,}
\def\ucl(#1){\lfloor #1 \rfloor}% up-closure
\def\dcl(#1){\lceil #1 \rceil}% down-closure
\def\specrel#1#2{\mathrel{\mathop{\kern0pt #1}\limits_{#2}}}
\def\restricts{\!\restriction\!}
\def\interior{\mathaccent"7017\relax}

\def\F{\mathcal F}
\newcommand\N{\mathcal N}
\renewcommand\O{\mathcal O}
\renewcommand\P{\mathcal P}
\newcommand\Q{\mathcal Q}

\newcommand\V{\mathcal V}

\def\lowfwd #1#2#3{{\mathop{\kern0pt #1}\limits^{\kern#2pt\raise.#3ex
\vbox to 0pt{\hbox{$\scriptscriptstyle\rightarrow$}\vss}}}}
\def\lowbkwd #1#2#3{{\mathop{\kern0pt #1}\limits^{\kern#2pt\raise.#3ex
\vbox to 0pt{\hbox{$\scriptscriptstyle\leftarrow$}\vss}}}}
\def\fwd #1#2{{\lowfwd{#1}{#2}{15}}}
\def\ve{\kern-1.5pt\lowfwd e{1.5}2\kern-1pt}
\def\vedash{{\mathop{\kern0pt e\lower.5pt\hbox{${}% logically \v(e')
     \scriptstyle'$}}\limits^{\kern0pt\raise.02ex
     \vbox to 0pt{\hbox{$\scriptscriptstyle\rightarrow$}\vss}}}}
\def\ev{\kern-1pt\lowbkwd e{0.5}2\kern-1pt}
\def\vf{\kern-2pt\lowfwd f{2.5}2\kern-1pt}
\def\vfdash{{\mathop{\kern0pt f\raise 1pt\hbox{${}% logically \v(f')
     \scriptstyle'$}}\limits^{\kern2pt\raise.02ex
     \vbox to 0pt{\hbox{$\scriptscriptstyle\rightarrow$}\vss}}}}
\def\fv{\lowbkwd f01}

\def\vr{\lowfwd r{1.5}2}
\def\rv{\lowbkwd r02}

\def\vrdash{{\mathop{\kern0pt r\lower.5pt\hbox{${}% logically \v(e')
     \scriptstyle'$}}\limits^{\kern0pt\raise.02ex
     \vbox to 0pt{\hbox{$\scriptscriptstyle\rightarrow$}\vss}}}}
\def\rvdash{{\mathop{\kern0pt r\lower.5pt\hbox{${}% logically \v(e')
     \scriptstyle'$}}\limits^{\kern0pt\raise.02ex
     \vbox to 0pt{\hbox{$\scriptscriptstyle\leftarrow$}\vss}}}}

\def\amgis{\lowbkwd \sigma02}
\def\vs{\lowfwd s{1.5}1}
\def\sv{\lowbkwd s{1.5}1}
\def\vsdash{{\mathop{\kern0pt s\lower.5pt\hbox{${}% logically \v(e')
     \scriptstyle'$}}\limits^{\kern0pt\raise.02ex
     \vbox to 0pt{\hbox{$\scriptscriptstyle\rightarrow$}\vss}}}}
\def\svdash{{\mathop{\kern0pt s\lower.5pt\hbox{${}% logically \v(e')
     \scriptstyle'$}}\limits^{\kern0pt\raise.02ex
     \vbox to 0pt{\hbox{$\scriptscriptstyle\leftarrow$}\vss}}}}
\def\vsddash{{\mathop{\kern0pt s\lower.5pt\hbox{${}% logically \v(e')
     \scriptstyle''$}}\limits^{\kern0pt\raise.02ex
     \vbox to 0pt{\hbox{$\scriptscriptstyle\rightarrow$}\vss}}}}
\def\vso{\lowfwd {s_0}11}

\def\vsidash{{\mathop{\kern0pt s_i\kern-3.5pt\lower.3pt\hbox{${}
     \scriptstyle'$}}\limits^{\kern0pt\raise.02ex
     \vbox to 0pt{\hbox{$\scriptscriptstyle\rightarrow$}\vss}}}}

\def\vsk{\lowfwd {s_k}11}

\def\vS{{\hskip-1pt{\fwd S3}\hskip-1pt}} %{{\vec S}} %

\def\vSr{{\vec S}_{\raise.1ex\vbox to 0pt{\vss\hbox{$\scriptstyle\ge\vr$}}}}
\def\vSstar{{\mathop{\kern0pt S\lower-1pt\hbox{$^*$}}\limits^{\kern2pt
     \vbox to 0pt{\hbox{$\scriptscriptstyle\rightarrow$}\vss}}}}
\def\vSdash{{\mathop{\kern0pt S\lower-1pt\hbox{${}% logically \v(e')
     \scriptstyle'$}}\limits^{\kern2pt\raise.1ex
     \vbox to 0pt{\hbox{$\scriptscriptstyle\rightarrow$}\vss}}}}
\def\vt{\lowfwd t{1.5}1}
\def\tv{\lowbkwd t{1.5}1}
\def\vT{{\fwd T1}}
\def\es{\emptyset}
\def\sub{\subseteq}
\def\supe{\supseteq}
\def\sm{\smallsetminus}
\def\td{tree-decom\-po\-si\-tion}

\newcommand\COMMENT[1]{}
\def\?#1{\vadjust{\vbox to 0pt{\vss\vskip-8pt\leftline{%
     \llap{\hbox{\vbox{\pretolerance=-1
     \doublehyphendemerits=0\finalhyphendemerits=0
     \hsize16truemm\tolerance=10000\small
     \lineskip=0pt\lineskiplimit=0pt
     \rightskip=0pt plus16truemm\baselineskip8pt\noindent
     \hskip0pt        %(without this, the first word is never hyphenated!)
     #1\endgraf}\hskip7truemm}}}\vss}}}

\title{Tree sets}
 \author{Reinhard Diestel}
 \date{}

\begin{document}
\abovedisplayshortskip=-3pt plus3pt
\belowdisplayshortskip=6pt

\maketitle

\begin{abstract}\noindent
  We study an abstract notion of tree structure which lies at the common core of various tree-like discrete structures commonly used in combinatorics: trees in graphs, order trees, nested subsets of a set, \td s of graphs and matroids etc.

Unlike graph-theoretical or order trees, these {\em tree sets\/} can provide a suitable formalization of tree structure also for infinite graphs, matroids, and set partitions. Order trees reappear as oriented tree sets.%
   \COMMENT{}

We show how each of the above structures defines a tree set, and  which additional information, if any, is needed to reconstruct it from this tree set.\looseness=-1
   \end{abstract}

\section{Introduction}\label{sec:intro}

There are a number of concepts in combinatorics that express the tree-likeness of discrete%
   \footnote{There are also non-discrete such concepts, such as $\Bbb R$-trees, which are not our topic here.}
   structures. Among these are:
 \begin{itemize}\itemsep=0pt
\item graph trees
\item order trees
\item nested subsets, or bipartitions, of a set.
 \end{itemize}

\medbreak\noindent
Other notions of tree-likeness, such as \td s of graphs or matroids, are modelled on these.%
   \COMMENT{}

All these notions of tree-likeness work well in their own contexts, but sometimes less well outside them:
 \begin{itemize}\itemsep=0pt
   \item graph trees need vertices, which in some desired applications~-- even as close as matroids~-- may not exist;%
   \COMMENT{}
   \item order trees need additional poset structure which is more restrictive than the tree-likeness it implies;%
   \footnote{Finite order trees, for example, correspond to rooted graph trees, but there is nothing in their definition from which we can abstract so that what remains corresponds to the underlying unrooted graph tree.}
   \item nested sets of bipartitions require a ground set that can be partitioned, which does not exist, say, in the case of \td s of a graph;
   \item \td s of infinite graphs, which are modelled on graph trees, cannot describe separations that are limits of other separations, because graph trees do not have edges that are limits of other edges.
 \end{itemize}

\medbreak 

The purpose of this paper is to study an abstract notion of tree structure which is general enough to describe all these examples, yet substantial enough that each of these instances, in their relevant context, can be recovered from it.%
   \COMMENT{}

We shall introduce this abstract notion of `tree sets' formally in Section~\ref{sec:separations}. It builds on a more general notion of `abstract separation systems' developed in~\cite{AbstractSepSys}. That paper consists of no more than some basic notions and facts we need here anyway, and is thus required as preliminary reading.%
   \footnote{Reference~\cite{AbstractSepSys} started life as the preliminary sections of this paper, and it should be read first. The reason it was split off is that abstract separation systems have since been used in several other papers too, and will be in more to come. So it seemed sensible to have the basics collected together in one place.}
   In a nutshell, a {\em separation system\/} will just be a poset with an order-reversing involution, and {\em tree sets\/} will be%
   \COMMENT{}
    nested separation systems: sets in which every element is comparable with every other element or its inverse.

Tree-likeness has been modelled in many ways~\cite{WoessTreeLike}, and even the idea to formalize it in this abstract way is not new. Abstract nested separation systems as above were introduced by Dunwoody~\cite{DunwoodyProtrees93, DunwoodyProTrees97} under the name of `protrees', as an abstract structure for groups to act on, and used by Hundertmark~\cite{profiles} as a basis for structure trees of graphs and matroids that can separate their tangles and related substructures. They have not, however, been studied systematically~-- which is our purpose in this paper.

Although studying abstract tree sets may seem amply justified by their ubi\-qui\-ty in different contexts, there are two concrete applications that I would like to point out. The first of these is to order trees. These are often used in infinite combinatorics to capture tree structure wherever it arises. The reason they can do this better than graph-theoretical trees is that they may contain limit points, as those tree-like structures to be captured frequently do. However, order trees come with more information than is needed just to capture tree structure, which can make their use cumbersome. For a finite tree structure, for example, they correspond to a graph-theoretical tree together with the choice of a root. A~change of root will change the induced tree order but not the underlying graph tree, which already captures that finite tree structure.

It turns out that abstract tree sets can provide an analogue of this also for infinite order trees: these will correspond precisely to the {\em consistently oriented\/} tree sets. Just as different choices of a root turn the same graph-theoretical tree into related but different order trees, different consistent orientations of an abstract trees set yield related but different order trees.
Order trees thus appear as something like a category of `pointed tree sets',%
   \footnote{More precisely: tree sets plus a consistent orientation. The latter can be specified by a `point'~-- i.e., an element of the set~-- only when the order tree has a least element.}
   not only when their tree structure is represented by a graph but always, also when they have limit points. It thus becomes possible to `forget' the ordering of an order tree but retain more than a set: the set plus exactly the information that makes it tree-like.\looseness=-1

The application of tree sets that originally motivated this paper was one to graphs, as follows.
In graph minor theory there are duality theorems saying that a given finite graph either has a certain highly connected substructure, such as a bramble, or if not then this is witnessed by a \td\ showing that such a highly connected substructure cannot exist, because `there is no room for it'~\cite{TangleTreeGraphsMatroids, ST1993GraphSearching}. The edges of the decomposition tree then correspond to separations of this graph that form a tree set in our sense.

Conversely, a tree set of separations of a finite graph or matroid is always induced by a \td\ in this way. Thus, in finite graphs and matroids, \td s and tree sets of separations amount to the same thing.

But for infinite graphs, tree sets of separations are more powerful: they need not come from \td s, since separations can have limits but edges in decomposition trees do not. This is why width duality theorems for infinite graphs, such as those in~\cite{duality1inf}, require tree sets of separations, rather than \td s, to express the tree structures that witness the absence of highly connected substructures such as tangles or brambles. See~\cite{ProfiniteTreeSets} for more on this.

Highly cohesive substructures can be identified not only in graphs, but also in much more general combinatorial structures: all those that come with a sensible notion of `separation', and hence give rise to abstract separation systems. Their highly cohesive substructures then take the form of {\em tangles\/} of these separation systems: orientations of their separations that are consistent in various ways (all including the basic consistency considered for tree sets in this paper) that define the differences between these types of tangle~\cite{ProfileDuality, TangleTreeAbstract, TangleTreeGraphsMatroids, MonaLisa}. We thus obtain a wealth of duality theorems for potentially very different combinatorial structures all based on the abstract tree sets studied in this paper.

In Section~\ref{sec:separations} we provide just the formal definitions needed to state our theorems; for all the basic facts about abstract separation systems that we shall need in our proofs we refer the reader to~\cite{AbstractSepSys}. In Sections~\ref{sec:Gtrees}, \ref{sec:Otrees}, \ref{sec:Bipartitions} and~\ref{sec:Strees}, respectively, we show how abstract tree sets can be used to describe the tree structures of our earlier examples: of graph-theoretical trees, of order trees, of nested sets of bipartitions of a set, and of \td s of graphs and matroids. We shall also see how these representations of tree sets can be recovered from the tree sets they represent. Where relevant we shall point out how, conversely, abstract tree sets describe tree-like structures in these contexts that do not come from such examples: where tree sets provide not just a convenient common language for different kinds of tree structures but define new ones, including new ones that are needed for applications in traditional settings such as graphs and matroids~\cite{duality1inf}.

Any terminology used but not defined in either~\cite{AbstractSepSys} or this paper can be found in~\cite{DiestelBook16}.

\section{Abstract separation systems and tree sets%
   \protect\footnote{This section is provided only to make this paper formally self-contained; I~encourage the reader to read~\cite{AbstractSepSys} instead of this section, or at least to refer to~\cite{AbstractSepSys} in parallel.}%
   }\label{sec:separations}

A \emph{separation of a set} $V$ is a set $\{A,B\}$ such that $A\cup B=V$.%
   \footnote{We can make further requirements here that depend on some structure on~$V$ which $\{A,B\}$ is meant to separate. If $V$ is the vertex set of a graph~$G$, for example, we usually require that $G$ has no edge between $A\sm B$ and $B\sm A$. But such restrictions will depend on the context and are not needed here; in fact, even the separations of a {\em set}~$V$ defined here is just an example of the more abstract `separations' we are about to introduce.}
   The ordered pairs $(A,B)$ and $(B,A)$ are its {\it orientations\/}. The {\em oriented separations\/} of~$V$ are the orientations of its separations. Mapping every oriented separation $(A,B)$ to its {\it inverse\/} $(B,A)$ is an involution%
   \COMMENT{}
   that reverses the partial ordering \[(A,B)\le (C,D) :\Leftrightarrow A\subseteq C \text{ and } B\supseteq D,\]
since the above is equivalent to $(D,C)\le (B,A)$. Informally, we think of $(A,B)$ as \emph{pointing towards}~$B$ and \emph{away from}~$A$.

More generally, a {\em separation system\/} $(\vS,\le\,,\!{}^*)$ is a partially ordered set $\vS$ with an order-reversing involution\,*. Its elements are called {\em oriented separations\/}. An {\em isomorphism\/} between two separation systems is a bijection between their underlying sets that respects both their partial orderings and their involutions.

When a given element of $\vS$ is denoted as~$\vs$, its {\em inverse\/}~$\vs^*$ will be denoted as~$\sv$, and vice versa. The assumption that * be {\em order-reversing\/} means that, for all $\vr,\vs\in\vS$,
\begin{equation}\label{invcomp}
\vr\le\vs\ \Leftrightarrow\ \rv\ge\sv.
\end{equation}

A {\em separation\/} is a set of the form $\{\vs,\sv\}$, and then denoted by~$s$. We call $\vs$ and~$\sv$ the {\em orientations\/} of~$s$. The set of all such sets $\{\vs,\sv\}\sub\vS$ will be denoted by~$S$. If $\vs=\sv$, we call both $\vs$ and $s$ {\em degenerate\/}.

When a separation is introduced ahead of its elements and denoted by a single letter~$s$, we shall use $\vs$ and~$\sv$ (arbitrarily)%
   \COMMENT{}
   to refer to its elements. Given a set $S'$ of separations, we write $\vSdash := \bigcup S'\sub\vS$%
   \COMMENT{}
   for the set of all the orientations of its elements. With the ordering and involution induced from~$\vS$, this is again a separation system.%
   \COMMENT{}%
   \footnote{When we refer to oriented separations using explicit notation that indicates orientation, such as $\vs$ or $(A,B)$, we sometimes leave out the word `oriented' to improve the flow of words. Thus, when we speak of a `separation $(A,B)$', this will in fact be an oriented separation.}

Separations of sets, and their orientations, are clearly an instance of this abstract setup if we identify $\{A,B\}$ with $\{(A,B),(B,A)\}$.

A separation $\vr\in\vS$ is {\em trivial in~$\vS$\/}, and $\rv$ is {\em co-trivial\/}, if there exists $s \in S$ such that $\vr < \vs$ as well as $\vr < \sv$.%
   \COMMENT{}
   We call such an $s$ a {\em witness\/} of $\vr$ and its triviality. If neither orientation of~$r$ is trivial, we call~$r$ {\em nontrivial\/}.%
   \COMMENT{}

Note that if $\vr$ is trivial in~$\vS$ then so is every $\vrdash \le \vr$. If $\vr$ is trivial, witnessed by~$s$, then $\vr < \vs < \rv$ by~\eqref{invcomp}. Hence if $\vr$ is trivial, then $\rv$ cannot be trivial. In particular, degenerate separations are nontrivial.%
   \COMMENT{}

There can also be separations $\vs$ with $\vs < \sv$ that are not trivial.%
   \COMMENT{}
   But any\-thing smaller than these is again trivial: if $\vr < \vs\le\sv$, then $s$ witnesses the triviality of~$\vr$. Separations~$\vs$ such that $\vs\le\sv$, trivial or not, will be called {\em small\/}; note that, by~\eqref{invcomp}, if $\vs$ is small then so is every~$\vsdash\le\vs$.

The trivial oriented separations of a set~$V\!$, for example, are those of the form $\vr = (A,B)$ with $A\sub C\cap D$ and $B\supe C\cup D = V$ for some $s = \{C,D\}\ne r$ in the set $S$ considered. The small separations $(A,B)$ of $V$ are all those with $B=V$.\looseness=-1

\begin{DEF}
 A~separation system is {\em regular\/} if it has no small elements.%
   \COMMENT{}
   It is {\em essential\/} if it has neither trivial elements nor degenerate elements.
\end{DEF}

Note that all regular separation systems are essential.%
   \COMMENT{}

\begin{DEF}
The {\em essential core\/} of a separation system~$\vS$ is the essential separation system~$\vSdash$ obtained from~$\vS$ by deleting all its separations that are degenerate, trivial, or co-trivial in~$\vS$.%
   \COMMENT{}
\end{DEF}

An essential but irregular separation system can be made regular by deleting all pairs of the form $(\vs,\sv)$ from the relation~$\le$ viewed as a subset of~$\vS{}^2$: the triple  $(\vS,\le',\!{}^*)$, where $\vr<'\!\vs$ if and only if $\vr<\vs$ and $r\ne s$, is a regular separation system~\cite{AbstractSepSys}. We call it the {\em regularization\/} of~$\vS$.

A set $O\sub \vS$ of oriented separations is {\em antisymmetric\/} if $|O\cap \{\vs,\sv\}| \le 1$ for all $\vs\in\vS$: if $O$ does not contain the inverse of any of its nondegenerate elements.%
   \COMMENT{}
We call~$O$ \emph{consistent} if there are no distinct $r,s\in S$ with orientations $\vr < \vs$ such that $\rv,\vs\in O$.%
   \COMMENT{}

\medbreak

Two separations $r,s$ are {\em nested\/} if they have comparable orientations; otherwise they \emph{cross}.
Two oriented separations $\vr,\vs$ are {\em nested\/} if $r$ and~$s$ are nested.%
   \footnote{Terms introduced for unoriented separations may be used informally for oriented separations too if the meaning is obvious, and vice versa.}%
   \COMMENT{}
   We say that $\vr$ {\em points towards\/}~$s$, and $\rv$ {\em points away from\/}~$s$, if $\vr\le\vs$ or $\vr\le\sv$.

In this informal terminology, two oriented separations are nested if and only if they are either comparable or point towards each other or point away from each other. And a set $O\sub\vS$ is consistent if and only if it does not contain orientations of distinct separations that point away from each other.

A~set of separations is {\em nested\/} if every two of its elements are nested. 

\begin{DEF}
A {\em tree set\/} is a nested essential separation system. An {\em isomorphism of tree sets\/} is an isomorphism of separation systems that happen to be tree sets.
\end{DEF}

\noindent
The essential core of a nested separation system~$\vS$ is the tree set {\em induced by}~$\vS$.

\begin{DEF}
A {\em star (of separations)\/} is a set $\sigma$ of nondegenerate oriented separations whose elements point towards each other:  $\vr\le\sv$ for all distinct $\vr,\vs\in\sigma$.
\end{DEF}

 We allow $\sigma=\es$. Note that stars of separations are nested. They are also consistent: if distinct $\rv,\vs$ lie in the same star we cannot have $\vr < \vs$, since also $\vs\le \vr$ by the star property.

A star $\sigma$ is {\em proper\/} if, for all distinct%
   \COMMENT{}
   $\vr,\vs\in\sigma$, the relation $\vr\le \sv$ required by the definition of `star' is the only one among the four possible relations between orientations of distinct $r$ and~$s$: if $\vr \le \sv$ but $\vr\not\le\vs$ and $\vr\not\ge\vs$ and $\vr\not\ge\sv$.%
   \COMMENT{}

Our partial ordering on~$\vS$ also relates its subsets, and in particular its stars: for $\sigma,\tau\sub\vS$ we write $\sigma\le\tau$ if for every $\vs\in\sigma$ there exists some $\vt\in\tau$ with $\vs\le\vt$.%
   \COMMENT{}
    This relation is obviously reflexive and transitive, but in general it is not antisymmetric: if $\sigma$ contains separations $\vs < \vt$, then for $\tau = \sigma\sm\{\vs\}$ we have $\sigma < \tau < \sigma$ (where $<$ denotes `$\le$ but not~=').%
   \COMMENT{}
   However, it is antisymmetric on antichains, and thus in particular on proper stars~\cite{AbstractSepSys}.

We call a star $\sigma\in\vS$ {\em proper in~$\vS$} if it is proper and is not a singleton~$\{\sv\}$ with $\sv$ co-trivial in~$\vS$. We shall call such stars {\em co-trivial singletons\/}.

When we speak of {\em maximal proper stars in\/} a separation system~$(\vS,\le\,,\!{}^*)$, we shall always mean stars that are $\le$-maximal in the set of stars that are proper in~$\vS$. Maximal stars in tree sets will play a key role in describing their structure.

We refer to~\cite{AbstractSepSys} for examples and various properties of stars in separation systems that we shall use throughout this paper.

\medbreak

An \emph{orientation} of a separation system~$\vS$,%
   \COMMENT{}
   or of a set~${S}$ of separations, is a set $O\sub{\vS}$ that contains for every $s\in{S}$ exactly one of its orientations $\vs,\sv$. A~\emph{partial orientation} of~${S}$ is an orientation of a subset of~${S}$: an antisymmetric subset of~$\vS$. We shall be interested particularly in consistent orientations.% 
   \COMMENT{}

Every consistent orientation $O$ of a regular separation system~$\vS$ is the {\em down-closure\/}
 $$\dcl(\sigma)_{\vS} := \{\, \vr\in\vS\mid\exists \vs\in\sigma\colon \vr\le\vs\,\}$$
 in~$\vS$ of the set $\sigma$ of its maximal elements~-- provided that every element of~$O$ lies below some maximal element (which can fail when $\vS$ is infinite). If~$\vS$ is a tree set, then these sets~$\sigma$ are its maximal proper stars; we call them the {\em splitting stars\/} of~$\vS$.

See~\cite{AbstractSepSys} for proofs of these assertions and further background needed later.

\section{Tree sets from graph-theoretical trees}\label{sec:Gtrees}

The set
 $$\vec E(T) := \{\, (x,y) : xy\in E(T)\,\}$$
 of all \emph{orientations} $(x,y)$ of the edges $xy=\{x,y\}$ of a tree~$T$ form a regular tree set with respect to the involution $(x,y)\mapsto (y,x)$ and the \emph{natural partial ordering} on~$\vec E(T)$: the ordering in which $(x,y) < (u,v)$ if $\{x,y\}\ne\{u,v\}$ and the unique $\{x,y\}$--$\{u,v\}$ path in $T$ joins $y$ to~$u$. We call this is the {\it edge tree set\/} $\tau(T)$ of~$T$. For every node~$t$ of~$T$, we call the set
   $$\vec F_t := \{(x,t) : xt\in E(T)\}$$
 of edges at~$t$ and oriented towards~$t$ the {\em oriented star at~$t$} in~$T$.

\begin{LEM}\label{splittingstars}
The sets~$\vec F_t$ are the splitting stars of the edge tree set~$\tau(T)$ of~$T$.%
   \COMMENT{}
\end{LEM}%
   \COMMENT{}

\proof Let $T$ be any tree. The down-closure in $\tau = \tau (T)$ of a set $\vec F_t\sub \vec E(T)$ is clearly a consistent orientation of~$E(T)$ whose set of maximal elements is precisely~$\vec F_t$.

Conversely, let ${\sigma\sub\vec E(T)}$ split~$\tau$. Then $\sigma$ is the set of maximal elements of some consistent orientation $O$ of~$E(T)$, and $O\sub\dcl(\sigma)$. In particular, $\sigma\ne\es$%
   \COMMENT{}
   unless $E(T)=\es$ (in which case the assertion is true),%
   \COMMENT{}
  so $O$ has a maximal element~$(x,t)$. For every neighbour $y\ne x$ of~$t$, the maximality of $(x,t)$ in~$O$ implies that $(t,y)\notin O$ and hence $(y,t)\in O$.

Thus, $\vec F_t\sub O$. As $O$ is closed down in~$\tau$ and the down-closure of $\vec F_t$ in~$\tau$ orients all of~$E(T)$, this down-closure equals~$O$ and has $\vec F_t$ as its set of maximal elements, giving $\sigma = \vec F_t$ as desired.
\endproof

Lemma~\ref{splittingstars} allows us to recover a tree $T$ from its edge tree set~$\tau$. Indeed, given just~$\tau$, let $V$ be the set of its splitting stars~$\sigma$. Define a graph $G$ on~$V$ by taking $\tau$ as its set of oriented edges and assigning to every edge $\vs$ the splitting star of~$\tau$ containing it as its terminal node. These are well defined~-- i.e., every $\vs\in\tau$ lies in a unique splitting star~-- by Lemma~\ref{splittingstars} and our assumption that $\tau = \tau(T)$. Then, clearly, the map $t\mapsto \vec F_t$ is a graph isomorphism between $T$ and~$G$.

Our assumption above that $\tau$ is the edge tree set of {\em some\/} tree was used heavily in the argument above. And indeed, it cannot be omitted altogether: an arbitrary tree set need not be realizable as the edge tree set of a tree.%
   \COMMENT{}
   Finite regular tree sets are, though, and our next aim is to prove this in Theorem~\ref{Trees}.

Given a tree set~$\tau$,%
   \COMMENT{}
   write $\O=\O(\tau)$ for the set of its consistent orientations. Define a directed graph $\vT$ with edge set $E(\vT) = \tau$ as follows. For each $\vs\in\tau$ there is a unique $t\!_{\vs}\in\O$ in which $\vs$~is maximal, by \cite[Lemma~4.1\,(iii)]{AbstractSepSys} applied with $P=\{\vs\}$. Let $\vT = \vT(\tau)$ be the directed graph on $\{\,t\!_{\vs}\mid \vs\in\tau\,\}\sub\O$ with edge set~$\tau$, where $\vs\in\tau$ runs from $t\!_{\sv}$ to~$t\!_{\vs}$. Note that these are distinct,%
   \COMMENT{}
   since $\vs\ne\sv$ as $\tau$ has no degenerate elements. Let $T = T(\tau)$ be the underlying undirected graph, with pairs $\vs,\sv$ of directed edges identified into one undirected edge~$s$ inheriting its orientations $\vs,\sv$ from~$\vT$.%
   \COMMENT{}

Thus, $\vec E(T) = E(\vT)$. In fact, let us remember for the definition of~$\vec E(T) = \tau(T)$ also the information from~$E(\vT)$ of which orientation of an edge $s$ of~$T$ is $\vs$ and which is~$\sv$. Unlike with arbitrary tree sets, the elements of tree sets of the form $\tau(T(\tau'))$ thus come with fixed names: those inherited from~$\tau'$ as defined above.

\begin{LEM}\label{Ttau}
For every finite tree set~$\tau$,%
   \COMMENT{}
   the graph $T(\tau)$ is a tree on~$\O(\tau)$.
\end{LEM}

\begin{proof}
   For every node $t\in T=T(\tau)$, the set $\sigma_t$%
   \COMMENT{}
   of its incoming edges is precisely the set of all $\vs\in\tau$ that are maximal in the orientation $t$ of~$\tau$.%
   \COMMENT{}
   As $\tau$ is finite, $t$~is the down-closure of%
   \COMMENT{}
   its maximal elements, so these~$\sigma_t$ are splitting stars of~$\tau$.%
   \COMMENT{}
   Conversely, every set $\sigma$ splitting~$\tau$ is clearly of this form: pick $\vs\in\sigma$, and notice that $\sigma = \sigma_t$ for $t:= t\!_{\vs}$ by \cite[Lemma~4.1\,(iii)]{AbstractSepSys}.%
   \COMMENT{}

Similarly, the finiteness of~$\tau$ implies by \cite[Lemma~4.2]{AbstractSepSys} that every $O\in\O$ is the downclosure of its maximal elements, so $O=t\!_{\vs}$ for every $\vs$ maximal in~$O$, giving $V(T)=\O(\tau)$.%
   \COMMENT{}
   Let us now prove that $T$ is a tree.

We noted before that $t\!_{\vs}\ne t\!_{\sv}$ for all~$\vs\in\tau$, so $T$ has no loops. In fact, $T$~is acyclic. Indeed, if $\vso,\dots,\vsk$ are the edges of an oriented cycle in~$\vT(\tau)$,%
   \COMMENT{}
   then each of these and the inverse of its (cyclic) successor lie in a common set~$\sigma_t$. Since these $\sigma_t$ are stars of separations \cite[Lemma~4.5]{AbstractSepSys}, we have $\vso < \ldots < \vsk < \vso$ with a contradiction.%
   \COMMENT{}

To see that $T$ is connected, let $t,t'$ be nodes in different components, with $|t\cap t'|$ maximum. Let $\vs$ be maximal in $t\sm t'$ (which we may assume is non-empty).%
   \COMMENT{}
   Then $\vs$ is maximal also in~$t$: any $\vsdash\in t$ greater than~$\vs$ would also lie in~$t'$, and hence so would~$\vs$ by the consistency of~$t'$ (which also orients~$s$). Replacing $\vs$ in $t$ with~$\sv$ therefore changes $t$ into an orientation~$t''$ of $\tau$ that is again consistent, by the maximality of $\vs$ in~$t$.%
   \COMMENT{}
   In~$t''$ the separation $\sv$ is maximal: for any $\rv > \sv$ we have $\vr < \vs\in t$, so $\vr\in t$ by the consistency of~$t$ and hence also $\vr\in t''$,%
   \COMMENT{}
   giving $\rv\notin t''$. Hence $s=tt''$, and in particular $t''$ lies in the same component of~$T$ as~$t$. Since it agrees with $t'$ on more separations than $t$ does, we have a contradiction to the choice of $t$ and~$t'$.
\end{proof}

Our aim was to show that every finite tree set is the edge tree set of some tree. In order for Lemma~\ref{Ttau} to imply this we still need to know that $\tau$ coincides with $\vec E (T(\tau))$ not only as a set, which it does by definition, but also as a poset: that the $\tau$ is indeed the edge tree set of~$T(\tau)$.
Part~(i) of the Theorem~\ref{Trees} below makes this precise.

Theorem~\ref{Trees}\,(ii) says that the tree $T(\tau)$ whose edge tree set represents~$\tau$, as provided by~(i), is unique up to a canonical graph isomorphism: if $\tau$ is the edge tree set also of some other tree~$T$, then that tree~$T$ is isomorphic to~$T(\tau)$ by an isomorphism that can be defined just in terms of~$\tau$.%
   \COMMENT{}

Given a tree~$T$, write $O_t$ for the orientation of~$\tau(T)$ that orients every edge of~$T$ towards~$t$.

\begin{THM}\label{Trees}
\begin{enumerate}[\rm (i)]
\item For every finite regular tree set~$\tau'$, the identity is a tree set isomorphism between $\tau'$ and~$\tau = \tau(T(\tau'))$.
\item For every finite tree~$T$, the map $t\mapsto O_t$ is a graph isomorphism between $T$ and~$T(\tau(T))$.
\end{enumerate}
\end{THM}

Note that if (ii) is applied to a tree of the form $T=T(\tau')$, then $t\mapsto O_t$ is the identity on $V(T) = \O(\tau') = \O(\tau(T(\tau')))$, by~(i) applied to~$\tau'$.%
   \COMMENT{}

\begin{proof}
(i) The fact that elements $\vs,\sv$ of~$\tau'$ are inverse to each other also in $\tau$ was built into the definition of $(\tau =)\ \vec E(T(\tau')) = E(\vT(\tau'))$.

It remains to show that elements $\vr,\vs\in\tau'$ satisfy $\vr < \vs$ in~$\tau'$ if and only if they do so in~$\tau$. Since $\tau'$ and~$\tau$ are both regular,%
   \COMMENT{}
   $\vr < \vs$~in either of them implies that $r\ne s$, so we may assume this. The equivalence of the two assertions%
   \COMMENT{}
   follows by induction on the length~$\ell$ of the unique $r$--$s$ path in~$T(\tau')$, using the transitivity of~$\le$,%
   \COMMENT{}
   once we have shown it for $\ell=0$, that is, when $r$ and~$s$ share a vertex~$t$ of~$T(\tau')$.

Since $r\ne s$ by assumption, this means that the consistent orientation $t$ of~$\tau'$ has distinct maximal elements that are orientations of $r$ and~$s$, respectively. These form a proper star in~$\tau'$, because they are both maximal in~$t$ and $t$ is consistent and antisymmetric (and $\tau'$ is regular). And they form a proper star in~$\tau$, because they are both oriented towards and incident with~$t$.%
   \COMMENT{}
   Since, in any separation system, distinct elements of a proper star and their inverses are each related as required by the star property and not in any other way, the orientations of $r$ and~$s$ are related in~$\tau'$ as they are in~$\tau$.

(ii) For each $t\in T$, the orientation $O_t$ of~$\tau=\tau(T)$ is clearly consistent, so it is an element of~$\O(\tau) = V(T(\tau))$. As $O_t$ and~$O_{t'}$ differ on every edge of~$T$ between $t$ and~$t'$, the map is injective. To see that it is surjective, recall from \cite[Lemma~4.2]{AbstractSepSys} that the elements of~$\O(\tau)$ are precisely the down-closures of the subsets splitting~$\tau$, which by Lemma~\ref{splittingstars} are the sets~$\vec F_t$. But the down-closure of~$\vec F_t$ is precisely~$O_t$. Thus, our map $t\mapsto O_t$ is a bijection from~$V(T)$ to~$\O(\tau)$.

To see that $t\mapsto O_t$ is a graph isomorphism, notice that for any edge ${tt'\in T}$ its orientation~$\vs$ from $t$ to~$t'$ is maximal in~$O_{t'}$, while its other orientation~$\sv$ is maximal in~$O_t$. Hence $s$ is an edge of $T(\tau)$ between its vertices $O_t$ and~$O_{t'}$.

For the converse note first that, given adjacent vertices $O_t$ and~$O_{t'}$ of~$T(\tau)$, the edge~$s$ joining them in~$T(\tau)$ is the only edge of~$T$ which $O_t$ and~$O_{t'}$ orient differently. Indeed, by definition of~$T(\tau)$, the edge~$s$ has an orientation $\vs$ that is maximal in~$O_t$, and whose inverse~$\sv$ is maximal in~$O_{t'}$. By definition of $O_t$ and~$O_{t'}$, this means that $s$ is the edge $tt'$ of~$T$.%
   \COMMENT{}

Now if two vertices $t,t'\in T$ are {\em not\/} adjacent in~$T$,%
   \COMMENT{}
   then the $t$--$t'$ path in~$T$ contains distinct edges $r$ and~$s$. As $O_t$ and~$O_{t'}$ disagree on both these edges, they cannot be adjacent in~$T(\tau)$.%
   \COMMENT{}
   \end{proof}

Infinite tree sets need not be isomorphic to the edge tree set of a tree. Indeed, as discussed in the introduction, one of%
   \COMMENT{}
   our motivations for studying tree sets is that they can capture tree-like structures in infinite combinatorics that actual trees cannot represent.

For infinite tree sets~$\tau$ that can be represented by a tree~$T$, we can still use its splitting stars as the nodes of~$T$ (Lemma~\ref{splittingstars}) and define its edges as earlier. But note that $\tau$ may have consistent orientations that are not the down-closure of a splitting star, and so the nodes of~$T$ may correspond to only proper a subset of~$\O(\tau)$. Indeed, orienting the edges of~$T$ towards an end of~$T$ is consistent but such orientations have no maximal elements.

Gollin and Kneip~\cite{TreelikeSpaces} have shown that Theorem~\ref{Trees} extends to precisely those infinite tree sets that contain no chain of order type~$\omega+1$, and that all tree sets can be represented as edge tree set of `tree-like' topological spaces.

\section{Tree sets from order trees}\label{sec:Otrees}

An {\em order tree\/}, for the purpose of this paper, is a poset $(T,\le)$ in which the down set
 $$\interior{\dcl(t)} := \{\,s\in T\mid s < t\,\}$$
 below every element $t\in T$ is a chain. We do not require this chain to be well-\penalty-200 ordered. To ensure that order trees induce order trees on the subsets of their ground set, we also do not require that every two elements have a common lower bound. Order trees that do have this property will be called {\em connected\/}.%
   \COMMENT{}

Order trees are often used to describe the tree-likeness of other combinatorial structures. In such contexts it can be unfortunate that they come with more information than just this tree-likeness, and one has to find ways of `forgetting' the irrelevant additional information.

Theorem~\ref{OrderTrees} below offers a way to do this: it canonically splits the information inherent in an order tree into the `tree part' represented by an unoriented tree set, and an `orienting part' represented by an orientation of this tree set. These orientations will be consistent. Indeed, we shall see that order trees correspond precisely to consistently oriented tree sets, finite or infinite.

As in~Section~\ref{sec:Gtrees}, let us first define for a given order tree $T$ a tree set $\tau=\tau(T)$ and a consistent orientation~$O_T$ of~$\tau$, and then conversely for a given tree set~$\tau$ and any consistent orientation~$O$ of~$\tau$ an order tree~$T=T(\tau,O)$. Applying these two operations in turn, starting from either an arbitrary order tree or from an arbitrary consistently oriented tree set, will yield an automorphism of order trees or of tree sets~-- in fact, the identity or something as close to the identity as is formally possible.%
   \COMMENT{}

For the first part, let $T=(X,\le)$ be any order tree. Our aim is to extend~$T$ to a tree set, i.e., to find a tree set $(\vS,\le\,,{}^*)$ such that $(X,\le)$ is a subposet of~$(\vS,\le)$. So we have to add%
   \COMMENT{}
   some inverses. Let $X^* = \{\,x^*\mid x\in X\,\}$ be a set disjoint from~$X$ and such that $x^*\ne y^*$ whenever $x\ne y$. Extend $\le$ to~$X\cup X^*$ by letting

\medskip\indent\indent
   $x^* < y^*\quad $if and only if$\quad x > y$ in~$T$;

\smallskip\indent\indent
   $x^* < y~\hskip1.5pt\quad$if and only if$\quad x,y$ are incomparable in~$T$.

\medskip\noindent
For every $x\in X$ let $x^{**}=x$; this defines an involution ${}^*\colon x\mapsto x^*$ on~$X\cup X^*$. Let $\tau(T) := (X\cup X^*, \le\,, {}^*)$ and $O_T := X^*$.

\begin{LEM}\label{TauTOrder}
Whenever $T=(X,\le)$ is an order tree,%
   \COMMENT{}
   $\tau(T)$ is a regular tree set, and $O_T$ is a consistent orientation of~$\tau(T)$.
\end{LEM}

\begin{proof}
For a proof that $\tau(T)$ is a regular tree set, the only nontrivial claim to check is that $\le$ is transitive on~$X\cup X^*$.%
   \COMMENT{}

Consider any $x,y,z\in X$. Suppose first that $x^* < y < z$. The first inequality implies, by our definition of~$<$, that $x$~and $y$ are incomparable in~$X$. But then so are $x$ and~$z$ (giving $x^*<z$ as desired): if $x < z$ then $x,y < z$, which makes $x$ and~$y$ comparable (which they are not) since $X$ is an order tree, while if $z < x$ then $y < z < x$ in~$X$, again contradicting the incomparability of $x$ and~$y$. Similarly if $x^* < y^* < z$ then $z^* < y < x$, which as just seen implies $z^* < x$ and hence $x^* < z$.%
   \COMMENT{}
   This covers all cases that do not follow at once from the transitivity of~$\le$ on~$X$.%
   \COMMENT{}

Finally, $O_T = X^*$ is a consistent orientation of~$\tau(T)$, since $X\cap X^*=\es$%
   \COMMENT{}
   and we never have $x < y^*$ for any $x^*,y^*\in X^*$.
 \end{proof}

We remark that, if $T=(X,\le)$ is connected,%
  \COMMENT{}
  then $\tau(T)$ is in fact the unique smallest regular%
   \COMMENT{}
   tree set to which $T$ extends. Indeed, any tree set $\tau$ that induces~$T$ on a subset $X$ must also contain a set $X^*$ of inverses. Let us show that $X^*$ will be disjoint from~$X$ if $\tau$ is regular.%
   \COMMENT{}
   Suppose that $y,z\in X$ are such that $y^*=z$ (and hence $z^*=y$). If $y$ and~$z$ are comparable, with $y<z$ say, then this makes $y$ small, contradicting the regularity of~$\tau$. But if they are incomparable, there will be an $x\in X$ below both (since $X$ is connected), so our assumption of $y^*=z$ makes $x$ trivial (and hence small), a contradiction.%
  \COMMENT{}

The proof of our remark%
   \COMMENT{}
   will be completed by the following uniqueness lemma, which has the disjointness of $X$ and~$X^*$ built into its premise%
   \COMMENT{}
   and therefore holds also for disconnected order trees. Note that if $(X,\le)$ is connected and $\tau$ is as specified in the lemma, then $X^*$~is necessarily consistent in~$\tau$: otherwise there are $y,z\in X$ are such that $y < z^*$, and by the connectedness of~$X$ there exists $x\in X$ with $x\le y$ and~$x\le z$, so $x\le y < z^*\le x^*$ is small, contradicting the regularity of~$\tau$.

\begin{LEM}\label{UniqueExtension}
Let $\tau = (\vS,\le\,,{}^*)$ be a regular tree set. Let $X\sub\vS$ be antisymmetric%
   \COMMENT{}
   and such that $(X,\le)$ is an order tree. If $X^* = \{\,x^*\mid x\in X\,\}$ is consistent in~$\tau$, then $\tau$ induces $\tau(T)$ on~$X\cup X^*$.
\end{LEM}

\begin{proof}
The involutions in $\tau$ and in $\tau(T)$ agree by definition of~$X^*$.%
   \footnote{Recall that, formally, we did not specify $X^*$ precisely in the definition of~$\tau(T)$: we took `any' set $X^*$ disjoint from~$X$ and with a bijection $x\mapsto x^*$ from $X$ to~$X^*$. The intended reading of Lemma~\ref{UniqueExtension} is that this is now the set $\vS\sm X$, whose elements are specified as $x^*$ by the involution in~$\tau$. The fact that this agrees with the involution in~$\tau(T)$ is then  tautological.}

It remains to show that $\tau$ and in $\tau(T)$ define the same ordering on~$X\cup X^*$. Since $\tau$ and~$\tau(T)$ are regular, $x$~and $x^*$ are incomparable in both, for all $x\in X$. Now consider distinct $x,y\in X$. If $x$ and~$y$ are comparable, with $x>y$ say, we must have $x^* < y^*$ in~$\tau$ by~\eqref{invcomp}, in agreement with our definition of~$\tau(T)$. Assume now that $x$ and~$y$ are incomparable. The consistency of $X^*$ in~$\tau$ rules out that $x < y^*$. But since $\tau$ is nested, $x$~and $y$ must have comparable orientations. The only case left is that $x^*<y$,%
   \COMMENT{}
   as we defined it for~$\tau(T)$.
\end{proof}

Lemma~\ref{TauTOrder} showed us how to extend, canonically, a given order tree~$T$ to a regular tree set $\tau(T)$ in which its complement~$O_T$ is consistent. We now show that, conversely, deleting a consistent orientation from a regular tree set leaves an order tree.

Given a regular tree set~$\tau$ and a consistent orientation $O$ of~$\tau$, let $T(\tau,O)$ be the subposet $(X,\le)$ of~$\tau$ induced by~$X:=\tau\sm O$.

\begin{LEM}\label{TtauOrder}
Whenever $\tau=(\tau,\le\,,{}^*)$ is a regular tree set and $O$ is a consistent orientation of~$\tau$, the poset $T(\tau,O)$ is an order tree.
\end{LEM}

\begin{proof}
Note first that for $X:=\tau\sm O$ we have $X = \{\,\sv\mid \vs\in O\,\}$, since $\tau$ is regular. For our proof that $(X,\le)$ is an order tree, consider $\vr,\vs,\vt\in O$ with $\rv,\sv < \tv$, and let us show that $\rv,\sv$ are comparable in~$\tau$. If not, then $\vr$ is comparable with~$\sv$ (and $\rv$ with~$\vs$),%
   \COMMENT{}
   because $r$ and~$s$ have comparable orientations since $\tau$ is a tree set.%
   \COMMENT{}
   Since $O$ is consistent we cannot have $\vr > \sv$, so $\vr < \sv$. But this implies that $\vt < \vr < \sv < \tv$ with a contradiction, since $\tau$ has no small elements.
\end{proof}

We have seen that every order tree canonically gives rise to a consistently oriented tree set, and that every consistently oriented tree set canonically induces an order tree. Let us now show that these maps are, essentially, inverse to each other.

When we convert a given order tree $T=(X,\le)$ into an oriented tree set, $\tau=\tau(T)$, we start by adding a set $X^*$ of inverses disjoint from~$X$, and this set~$X^*$ will be the desired orientation~$O=O_T$ of~$\tau$. Converting $\tau$ and~$O$ back into an order tree $T(\tau,O)$ then just consists of deleting~$O$ from the poset~$\tau$. This takes us back not only to the original ground set $X$ of~$T$, but the ordering of $T$ on~$X$ is preserved in the back-and-forth process. Theorem~\ref{OrderTrees}\,(ii) below expresses this.

Going the other way is entails a small technical complication. When we convert a tree set~$\tau$, given with a consistent orientation~$O$, into an order tree $T(\tau)$ by deleting~$O$ from the poset~$\tau$, we cannot hope to get $\tau$ back if we then expand $T=T(\tau)$ {\em canonially\/} to a tree set~$\tau(T)$, because the canonical process of extending~$T$ has no knowledge of the actual set~$O$ we deleted. All we can hope for is that the set we add corresponds to~$O$ as naturally as possible, and this is what Theorem~\ref{OrderTrees}\,(i) will say. Let us express this formally.

Given a regular tree set $\tau' = (\tau',\le\,,{}^*)$ and a consistent orientation~$O$%
   \COMMENT{}
   of~$\tau'$ let ${\rm id}'_O\colon \tau'\to\tau$, where $\tau = \tau(T(\tau',O))$, be the identity id on~$\tau'\sm O$ and commute on~$O$ with the composition ${}^*\circ{\rm id}\circ{}^*$ whose two maps~${}^*$ are the involutions on $\tau$ and on~$\tau'$, respectively.%
   \COMMENT{}
   Call ${\rm id}'_O$ the {\em canonization\/} of $\tau'$ {\em given}~$O$.

\begin{THM}\label{OrderTrees}
\begin{enumerate}[\rm (i)]
\item Given a regular tree set~$\tau'$ and a consistent orientation $O$ of~$\tau'$, let $T:= T(\tau',O)$ and $\tau:= \tau(T)$. Then the canonization $\tau'\to\tau$ of~$\tau'$ given~$O$ is an isomorphism of tree sets that induces the identity on~$\tau'\sm O$ and maps $O$ to~$O_T$.
\item Given an order tree~$T' = (X,\le)$, the identity on~$X$ is an order isomorphism from $T'$ to~$T := T(\tau(T'),O_{T'})$.
\end{enumerate}
\end{THM}

\begin{proof}
(i) As $O$ is an orientation of~$\tau'$ which, being a tree set, has no degenerate elements, $\tau'$~is the disjoint union of~$O$ and~$T = \{\,\sv\mid\vs\in O\,\}$. The latter, with its ordering induced by~$\tau'$, is an order tree by Lemma~\ref{TtauOrder}. The assertion now follows from Lemma~\ref{UniqueExtension}.%
   \COMMENT{}

(ii) By Lemma~\ref{TauTOrder}, $\tau(T')$~is a regular tree set of which $O_{T'}$ is a consistent orientation. By definition, $\tau(T')$~induces the original ordering of~$T'$ on its%
   \COMMENT{}
   ground set~$X$. But~$T$, by its definition as~$T(\tau(T'),O_{T'})$, also induces this ordering%
   \COMMENT{}
   on~$X$. Hence $T$ and~$T'$ induce the same ordering on their common ground set~$X$.
\end{proof}

\section{Tree sets from nested subsets of a set}\label{sec:Bipartitions}

Let $X$ be a non-empty set. The power set $2^X\!$ of $X$ is a separation system%
   \COMMENT{}
   with respect to inclusion and taking complements in~$X$. It contains the empty set~$\es$ as a small element,%
   \COMMENT{}
   but every nested subset of $2^X\sm\{\es\}$ is a regular tree set.

For compatibility with our earlier notion of set separations, let us refer to subsets $A$ of~$X$ as special kinds of separations of~$X$: those of the form $(A,X\sm A)$. A {\em bipartition\/} of~$X$, then, is an ordered pair $(A,B)$ of disjoint non-empty subsets of~$X$ whose union is~$X$. The bipartitions of~$X$ form a separation system $\vS(X)$ with respect to their {\em natural ordering\/} $(A,B)\le (C,D)$  defined by $A\sub C$ and the involution $(A,B)\mapsto (B,A)$. This separation system has no small elements, so every nested symmetric subset is a regular tree set.

Conversely, every regular tree set can be represented by set bipartitions:

\begin{THM}\label{SimpleBipartitions}
Every regular tree is isomorphic to a tree set of bipartitions of~a~\rlap{set.}
\end{THM}

\begin{proof}
Given a regular tree set~$\tau$, let $X\!:=\tau$ and consider
 $$\vec N := \{\,(X\!_\vt,X\!_\tv)\mid \vt\in\tau\,\},$$
 where%
   \COMMENT{}
   $X\!_\vs$~consists of~$\vs$ and the elements of~$\tau$ strictly below~$\!\vs$ and their inverses (but not~$\sv$). The sets $X\!_\vs$ and~$X\!_\sv$ forming a pair in~$\vec N$ are disjoint because $\tau$ has no trivial elements, and have union all of~$\tau$ because $\tau$ is nested. Thus, $\vec N$~consists of bipartitions of~$X$, and in particular $\vec N\sub\vS(X)$. Clearly, $\vs\mapsto (X\!_\vs,X\!_\sv)$ is a bijection from $\tau$ to~$\vec N$ that commutes with the involutions on $\tau$ and~$\vS(X)$ and preserves their orderings. In particular, $\vec N$~is a tree set isomorphic to~$\tau$.
\end{proof}

The set $X=\tau$ in Theorem~\ref{SimpleBipartitions} is a little larger than necessary. This is best visible  when $\tau$ is the edge tree set $\tau(T)$ of a finite tree~$T$. Then the elements of $\tau$ correspond to (oriented) bipartitions of the vertices of~$T$. So we could represent $\tau$ by these oriented bipartitions of $X=V(T)$, which has about half as many elements as the set $\tau = \vec E(T)$ chosen for $X$ in our proof of Theorem~\ref{SimpleBipartitions}.

Section~\ref{sec:Gtrees} tells us how to generalize this idea to tree sets~$\tau$ that are not edge tree sets of a finite tree: the vertices of~$T$ in the example correspond to the consistent orientations of~$\tau$. So let us try to use these directly to form~$X$.

Given~$\tau$, let $\O=\O(\tau)$ be the set of consistent orientations of~$\tau$. Every $\vs\in\tau$ defines a bipartition $(A,B)$ of~$\O$: into the set $A = \O\!_{\sv}$ of consistent orientations of~$\tau$ containing~$\sv$ and the set $B = \O\!_{\vs}$ of those containing~$\vs$.%
   \COMMENT{}
    Note that this is indeed a bipartition of~$\O$;%
   \COMMENT{}
   in particular, $A$ and~$B$ are non-empty by \cite[Lemma~4.1\,(i)]{AbstractSepSys} applied to~$\{\sv\}$ and~$\{\vs\}$, respectively.%
   \COMMENT{}

The map
 $$f\colon \vs\mapsto (\O\!_{\sv},\O\!_{\vs})$$
 from $\tau$ to~$\vS(\O)$ respects the involutions (obviously) and the partial orderings on~$\tau$ and~$\vS(\O)$. Indeed, if $\vr<\vs$ then no consistent orientation of~$\tau$ containing~$\rv$ contains~$\vs$, so $\O\!_{\rv}\sub\O\!_{\sv}$.%
   \footnote{Note that we just used the regularity of~$\tau$: if $\vr$ is small, we can have $\vr<\vs=\rv$, in which case `both' $\rv$ and~$\vs$ can occur in the same consistent orientation of~$\tau$.}
   Conversely, let us show that if $\vr,\vs\in\tau$ are such that $\O\!_{\rv}\sub\O\!_{\sv}$, equivalently $\O\!_{\vs}\sub\O\!_{\vr}$, then $\vr\le\vs$. If not, then $\{\rv,\vs\}$ is consistent, and hence extends by \cite[Lemma~4.1\,(i)]{AbstractSepSys} to a consistent orientation of~$\tau$. This lies in $\O\!_{\vs}\sm\O\!_{\vr}$, contradicting our assumption.%
   \COMMENT{}

Let us show that $f$ is injective. Consider distinct $\vr,\vs\in\tau$. Swapping their names as necessary, we may assume that $\vr\not< \vs$. Then $\{\rv,\vs\}$ is consistent and hence, by \cite[Lemma~4.1\,(i)]{AbstractSepSys}, extends to a consistent orientation of~$\tau$. This lies in~$\O\!_{\vs}\sm\O\!_{\vr}$, so $\O\!_{\vs}\ne\O\!_{\vr}$ and hence $f(\vr)\ne f(\vs)$ as desired.

We have shown that~$\tau$ is isomorphic, as a separation system, to the image $\vec\N(\tau)$ of~$\tau$ in~$\vS(\O)$ under the map $f$ of separation systems. In particular, $\vec\N(\tau)$~is a tree set of set bipartitions isomorphic to~$\tau$:%
   \COMMENT{}

\begin{THM}\label{Bipartitions}
Given any regular tree set~$\tau$, the map $f\colon \vs\mapsto (\O\!_{\sv},\O\!_{\vs})$ from $\tau$ to~$\vec\N(\tau)$ is an isomorphism of tree sets.
\end{THM}

We have seen that every symmetric nested set $\vec N$ of bipartitions of a set~$X$ is a regular tree set~$\tau$, which can in turn be represented as a tree set $\vec\N$ of bipartitions of a~set, e.g., the set $\O$ of its consistent orientations. However, in the transition $\vec N\to\tau\to\vec\N$ we are likely to lose some information: we shall not be able to recover~$\vec N$ from~$\vec\N$, not even up to a suitable bijection between $X$ and~$\O$.%
   \COMMENT{}

One reason is that $X$ may be `too large', larger than~$\O$. This happens if $X$ has distinct elements $x,x'$ that are {\em indistinguishable\/} by~$\vec N$: if no partition in~$\vec N$ assigns $x$ and~$x'$ to different partition classes.%
   \footnote{This was the case for $X=\tau$ in our proof of Theorem~\ref{SimpleBipartitions}: if $\tau$ is the edge tree set of a tree~$T$, say, then for every node $t\in T$ the elements of~$\lowbkwd F02_t$ cannot be distinguished by any separation of the form $(X\!_\vs,X\!_\sv)$.%
   \COMMENT{}
   More generally, $\vec N$~cannot distinguish the inverses of the elements of any splitting star in~$\tau$.}
   For this does not happen in~$\O$: distinct $O',O''\in \O$ are always distinguished by a separation $(\O',\O'')\in\vec\N$.%
   \COMMENT{}
   Indeed, as $O'\ne O''$ there exists $\vs\in\tau$ with $\sv\in O'$ and $\vs\in O''$. Then $O'\in\O'$ but $O''\in\O''$ for $(\O',\O'') = f(\vs)$.

But $X$ can also be `too small', smaller than~$\O$. This happens if $\vec N$ has a consistent orientation that is not of the form $\{\,(A,B)\in\vec N\mid x\in B\,\}$ for any $x\in X$. For every consistent orientation of~$\vec\N$ does have this form: it is the image under~$f$ of a consistent orientation~$O$ of~$\tau$, and hence equal to $\{\,(\O',\O'')\in\vec\N\mid O\in\O''\}$.

\begin{EX}\rm
Let $X$ be the vertex set of a ray~$R$, let $\vec N$ be the set of bipartitions of~$X$ corresponding to the oriented edges of~$R$, and choose from every inverse pair of separations in~$\vec N$ the separation that corresponds to the oriented edge of~$R$ which points towards its tail. This is a consistent orientation of~$\vec N$ not of the form $\{\,(A,B)\in\vec N\mid x\in B\,\}$.
\end{EX}

\begin{EX}\label{3star}\rm
Let $\vec N=\tau(T)$ where $T$ is a 3-star, orient every edge towards the centre of that star, and now consider the star of set separations that this induces on the set $X$ of only the leaves $x_1,x_2,x_3$.%
   \COMMENT{}
   Once more, this is a consistent orientation $O$ of $\vec N$ that is not of the above form, since no leaf $x$ lies in $\{x_j,x_k\}$ for each of the three partitions $(\{x_i\},\{x_j,x_k\})$ of $X$ that form~$O$.
\end{EX}

However, if we assume for $\vec N$ these two properties that $\vec\N$ will invariably have, we can indeed recover it from~$\vec\N$ in the best way possible, namely, up to a specified bijection between the ground sets involved:

\begin{THM}\label{RecoverN}
Let $\vec N$ be a tree set of bipartitions of a set~$X$ such that
\begin{itemize}\itemsep=0pt
   \item for all distinct $x,y\in X$ there exists $(A,B)\in\vec N$ such that $x\in A$ and~${y\in B}$;%
  \COMMENT{}
   \item for every consistent orientation $O$ of~$\vec N$ there exists an $x\in X$ such that $O = \{\,(A,B)\in\vec N\mid x\in B\,\}$.
\end{itemize}
Consider any isomorphism $g\colon\vec N\to\tau$ of tree sets. Let $f\colon\tau\to\vec\N(\tau)$ be the tree set isomorphism from Theorem~{\rm\ref{Bipartitions}}. Then there is a bijection%
   \footnote{formally, an `isomorphism of sets'}
 $h\colon X\to\O(\tau)$ whose natural action on~$\vec N$%
   \COMMENT{}
   equals $f\circ g$. In this way, $\vec N$~is canonically isomorphic, given~$g$, to the tree set $\vec\N(\tau)$.
\end{THM}

\proof
Given $x\in X$, let $O_x = \{\,(A,B)\in\vec N\mid x\in B\,\}$; this is clearly a consistent orientation of~$\vec N$. Hence%
   \COMMENT{}
   $h\colon x\mapsto g(O_x)$ is a well defined map from $X$ to~$\O(\tau)$. It is injective by the first condition in the theorem, and surjective by the second. Its action on the subsets of $X$ therefore%
   \COMMENT{}
   maps partitions of $X$ to partitions of~$\O(\tau)$.%
   \COMMENT{}
   The induced action of $h$ on~$\vec N$%
   \COMMENT{}
   is easily seen to be equal to~$f\circ g$,%
   \COMMENT{}
   which is an isomorphism of tree sets by the choice of~$g$ and Theorem~{\rm\ref{Bipartitions}\,(ii)}.
\endproof

Theorem~\ref{Bipartitions} provides us with another representation of an abstract regular tree set $\tau$ as a tree set $\vec\N$ of bipartitions of a set, and Theorem~\ref{RecoverN} shows that this representation describes, up to isomorphisms of tree sets, all the representations of $\tau$ as a tree set $\vec N$ of bipartitions of a set $X$ that is not unnecessarily large (ie, constains no two elements indistinguishable by the tree set) but large enough that every consistent orientation of~$\vec N$ is induced by one of its elements~$x$ (i.e., has the form $O_x=\{\,(A,B)\in\vec N\mid x\in B\,\}$ for some $x\in X$).

This latter requirement is quite stringent: there are many natural tree sets of bipartitions of a set whose consistent orientations are not all induced by an element of that set. Among these are the bipartitions of the leaves of a finite tree defined by its 1-edge cuts, as in the 3-star of Example~\ref{3star}:

\begin{EX}\label{FiniteTreeLeaves}\rm
Consider the edge tree set of a finite tree~$T$. Every edge $\ve\in\vec E(T)$ defines a bipartition $(A,B)$ of the vertices of~$T$: into the set $B$ of vertices to which $\ve$ points and the set $A$ of vertices to which $\ev$ points. These bipartitions of $V(T)$ are nested, and the tree set they form is canonically isomorphic to the tree set~$\tau(T)$ via $\ve\mapsto (A,B)$. Now consider the bipartitions $(A',B')$ which these $(A,B)$ induce just on the set of leaves of~$T$. These~$(A',B')$, too, will be distinct for different edges $\ve\in\vec E(T)$ as long as $T$ has no vertex of degree~2, and they will be nested in the same way as the $(A,B)$ that defined them. So these bipartitions of the leaves of~$T$ will still form a tree set isomorphic to~$\tau(T)$, via $\ve\mapsto (A',B')$. In particular, we can recover $(A,B)$ from~$(A',B')$ from these isomorphisms, as $(A',B')\mapsto\ve\mapsto (A,B)$.%
   \footnote{In fact, combining Theorems~\ref{Trees} and~\ref{RecoverN} we can reconstruct the entire tree~$T$ from the set $L$ of its leaves and the tree set of bipartitions of~$L$ that its edges induce.}%
   \COMMENT{}
\end{EX}

Let us generalize this example to arbitrary regular tree sets~$\tau$: rather than implementing $\tau$ as a tree set of bipartitions of the entire set~$\O=\O(\tau)$, let us use a subset $\O'$ of~$\O$ which, if $\tau = \tau(T)$ for a finite tree~$T$, corresponds to the leaves of~$T$ (while $\O$ corresponds to its entire vertex set).

If $\tau$ is finite, $\O'$~will be the set of those consistent orientations of~$\tau$ that have a greatest element,%
   \COMMENT{}
   which are precisely the down-closures of the maximal elements of~$\tau$.%
   \COMMENT{}
   In general, define $\O'(\tau)$ as the set of all the {\em directed\/} elements~$O$ of~$\O(\tau)$: those such that for all $\vr,\vs\in O$ there exists $\vt\in O$ such that ${\vr,\vs\le\vt}$.%
   \footnote{If $\tau=\tau(T)$ for an infinite tree~$T$, the orientations in $\O'(\tau)$ will be those towards a leaf or towards an end of~$T$.}

\begin{LEM}\label{Odash}
Every element of a regular%
   \footnote{This assumption helps shorten the proof but is not necessary.}
    tree set~$\tau$ lies in some $O\in\O'(\tau)$.
\end{LEM}

\begin{proof}
Given $\vs\in\tau$, let $O$ be the down-closure in~$\tau$ of some maximal%
   \COMMENT{}
   chain $\gamma$ in~$\tau$ containing~$\vs$. Using the regularity of~$\tau$ it is easy to check that $O$~is consistent and antisymmetric.%
   \COMMENT{}

It remains to show that $O$ orients all of~$\tau$. By its maximality, $\gamma$~contains for every given $\vr\in\tau$ some $\vs$ such that $\vs\not < \vr$. Applying this to both $\vr$ and~$\rv$ we obtain some $\vs\in\gamma$%
   \COMMENT{}
   such that neither $\vr$ nor~$\rv$ lies above~$\vs$. But since $\tau$ is nested, one of $\vr,\rv$ is comparable with~$\vs$. It thus lies below~$\vs$, and hence in~$O=\dcl(\gamma)$.
   \end{proof}

Let $\O'\!\!_{\vs}:= \{\,O\in\O'\mid \vs\in O\,\}$. With the same proof as in Theorem~\ref{Bipartitions}, the map
 $$f'\colon \vs\mapsto (\O'\!\!_{\sv},\O'\!\!_{\vs})$$
   from $\tau$ to the set $\vS(\O')$ of bipartitions of~$\O'$ communtes with the involutions and respects the partial orderings of the separation systems $\tau$ and~$\vS(\O')$. In particular, $\vec\N'(\tau):= f'(\tau)$ is a tree set~, and $f'$ is an isomorphism of tree sets whenever it is injective.

In our Example~\ref{FiniteTreeLeaves} where $\tau$ is edge tree set of a finite tree, $f'$ was injective if (and only if) that tree has no vertex of degree~2. For arbitrary regular tree sets~$\tau$ let us say that $\tau$ {\em branches everywhere\/}, or is {\em ever-branching\/}, if it contains no $\sub$-maximal proper star of order~2.%
   \COMMENT{}

\begin{LEM}\label{everbranch}
The map $f'$ is injective if and only if $\tau$ branches everywhere.
\end{LEM}

\begin{proof}
Suppose first that $\tau$ does not branch everywhere. Let $\sigma=\{\vr,\vs\}\sub\tau$ be a $\sub$-maximal proper star of order~2. Then also $r\ne s$, since $\tau$ is regular. We show that $f'$ is not injective, by proving that $f'(\vr)=f'(\sv)$ and $f'(\vs) = f'(\rv)$.%
   \COMMENT{}

We have to check that every $O\in\O'$ contains either both $\vr$ and~$\sv$ or both $\vs$ and~$\rv$. By consistency and since $\sigma$ is a star, $O$~cannot contain~$\amgis = \{\rv,\sv\}$. But neither can it contain~$\sigma$. Indeed, suppose $\sigma\sub O$. Since $O$ is directed, there exists $\vt\in O$ such that $\vr,\vs\le\vt$. Since $\tau$ is regular, this implies that $\tv$ is neither $\vr$ nor~$\vs$. But then $\sigma\cup\{\tv\}$ is a star contradicting the maximality of~$\sigma$.

Conversely, assume that $\tau$ is ever-branching, and thus contains no $\sub$-maximal proper star of order~2. To show that $f'$ is injective, consider distinct $\vr,\vs\in\tau$. We shall find an $O\in\O'$ that contains one of these but not the other; then $f'(\vr)\ne f'(\vs)$ by definition of~$f'$.

If $\{\vr,\vs\}$ is inconsistent, pick any $O\in\O'$ containing~$\vr$ (which exists by Lemma~\ref{Odash}); then $\vs\notin O$ by the consistency of~$O$. If $\{\vr,\vs\}$ is a star, pick any $O\in\O'$ containing~$\sv$; this will also contain $\vr\le\sv$ but not~$\vs$. Finally, assume that $\vr < \vs$ (say).%
   \COMMENT{}
   Since $\tau$ is regular, $\sigma=\{\vr,\sv\}$ is a proper star of order~2 by \cite[Lemma~3.4\,(ii)]{AbstractSepSys}. As $\tau$ is ever-branching, $\sigma$~is not maximal, so $\tau$ contains a bigger star $\{\vr,\sv,\tv\}$. Then $\vr<\vt$ as well as~$\sv < \vt$. Use Lemma~\ref{Odash} to find an orientation $O\in\O'$ containing~$\vt$. By consistency, then, $O$~also contains both $\vr,\sv<\vt$. Hence $O$ contains $\vr$ but not~$\vs$, as desired.
   \end{proof}

We have thus proved that ever-branching regular tree sets can be represented as tree sets of bipartitions of their directed consistent orientations:

\begin{THM}\label{BipartitionsSparse}
Let $\tau$ be any regular tree set. The map $f'\colon \vs\mapsto (\O'\!\!_{\sv},\O'\!\!_{\vs})$ from $\tau$ to~$\vec\N'(\tau)$ is an isomorphism of tree sets if and only if it is injective, which it is if and only if $\tau$ branches everywhere.
\end{THM}

What about an analogue of Theorem~\ref{RecoverN} for $f'$ and ever-branching~$\tau$? As before, the tree set $\vec\N'(\tau)$ of bipartitions of~$\O'$ distinguishes every two elements of~$\O'$: they will differ on some~$s$ and hence be separated by both~$f'(\hskip-1pt\vs)$ and~$f'(\hskip-1pt\sv)$. Hence, as in Theorem~\ref{RecoverN}, $\vec N$ must likewise distinguish every two points of~$X$ if we wish to recover a copy of it on~$\O'$. Hence any maximal element of~$\vec N$ must be of the form $(X\sm\{x\},\{x\})$: it cannot be a separation $(A,B)$ with $|B|\ge 2$.%
   \COMMENT{}

The second premise in Theorem~\ref{RecoverN}, however, can now be weakened substantially: we shall only need it for directed consistent orientations $O$ of~$\vec N$, not for all its consistent orientations. But for these we do need this assumption (more precisely, for sufficiently many of them; cf.\ Example~\ref{T2} below): we shall require that every directed consistent orientation $O$ of~$\vec N$ have the form $O = \{\,(A,B)\in\vec N\mid x\in B\,\}$ for some $x\in X$.

To see that this is necessary, consider what happens in~$\vec\N'$. Every directed consistent orientation $\tilde\O$%
   \COMMENT{}
   of~$\vec\N'$ will correspond via~$f'$ to a consistent orientation $O$ of~$\tau$, so that $f'(O)=\tilde\O$. By definition of~$\O'\!\!_{\vs}$ this $O$ lies, for every~$\vs\in O$, in the right partition set of~$f'(\vs) = (\O'\!\!_{\sv},\O'\!\!_{\vs})$.%
   \COMMENT{}
   Hence
  $$\tilde\O = \{\,f'(\vs)\mid \vs\in O\,\} = \{\,(\P,\Q)\in\vec\N'\mid O\in\Q\,\}\,.$$
 Hence if $f'$ is to be an isomorphism of tree sets mapping $O$ to~$\tilde\O$, then $O$ must have the corresponding form of $O = \{\,(A,B)\in\vec N\mid x\in B\,\}$ for some $x\in X$.%
   \COMMENT{}

\begin{EX}\label{T2}\rm
Consider as $\tau$ the edge tree set of the infinite binary tree~$T$. The directed consistent orientations of~$\tau$ are those that orient all the edges of~$T$ towards some fixed end~$\omega$. Our aim is to represent~$\tau$ by bipartitions of the set $X=\Omega$ of ends of~$T$. This can clearly be done.

However, we do not need all the ends in our set~$X$: all we need is that $X$ contains an end from either side of every edge $e$ of~$T$, so that $e$ defines a bipartition of~$X$ into two non-empty subsets. (These partitions will differ for distinct~$e$, since $T$ branches everywhere.) In the usual topology on~$\Omega$, this requirement can be expressed by saying that $X$ must contain a dense subset of~$\Omega$. Conversely, any such~$X$ suffices to ensure that these bipartitions capture~$\tau$.
\end{EX}

As these separations already distinguish every two elements of~$X$, we shall no longer have to require this explicitly in order to make $h'$ injective. Also, we do not have to require explicitly, in order to make $h'$ surjective, that no consistent orientations of $\vec N$ other than those with a greatest element%
   \COMMENT{}
   be of the form $O_x = {\{\,(A,B)\in\vec N\mid x\in B\,\}}$:%
   \COMMENT{}
   since $(X\sm\{x\},\{x\})\in\vec N$, this separation will lie in~$O_x$ and thus be its greatest element.%
   \COMMENT{}

We have thus shown the following strengthening%
   \COMMENT{}
   of Theorem~\ref{RecoverN} for ever-branching tree sets:

\begin{THM}\label{RecoverNsparse}
Let $\vec N$ be an ever-branching tree set of bipartitions of a set~$X$ such that for every directed consistent orientation $O$ of~$\vec N$ there is a unique%
   \footnote{This condition, which replaces the first premise in Theorem~\ref{RecoverN}, is not a severe restriction: if there is more than one such~$x$ for a given~$O$, delete all but one of them from~$X$. Since $\vec N$ cannot distinguish the elements deleted from this~$x$, this will not affect the representations of other elements of~$\tau$ by~$\vec N$.}
   $x$ in~$X$ such that $O = \{\,(A,B)\in\vec N\mid x\in B\,\}$. Let $X'$ be the set of all those~$x$. Then the conclusion of Theorem~\ref{RecoverN} holds with $f'%
   \COMMENT{}
    \colon\tau\to\vec\N'$ instead of $f\colon\tau\to\vec\N$.%
   \COMMENT{}
\end{THM}

When $\tau$ is finite, the sets $X$ in Theorems~\ref{RecoverN} and~\ref{RecoverNsparse} are, in a sense, maximal%
   \footnote{except for duplication of elements $x\in X$ by additional $x'$ that $\vec N$ cannot distinguish from~$x$\looseness=-1}
   and minimal, respectively, for the existence of a tree set $\vec N$ of bipartitions of~$X$ that represents~$\tau$. While in Theorem~\ref{RecoverN} the set $X$ has enough elements $x$ to give every consistent orientation of $\vec N$ the form~$O_x$, this is the case in Theorem~\ref{RecoverNsparse} only for the orientations of~$\vec N$ that have a greatest element, where it cannot be avoided. When $\tau$ is infinite, the set $X$ from Theorem~\ref{RecoverNsparse} need not be maximal, as seen in Example~\ref{T2}.

If desired, however, we can have any mixture of these extremes that we like. Indeed, starting with~$\tau$ we can build~$X$ by assigning to every $O\in\O$ a set~$X_O$ that is either empty or a singleton~$\{x_O\}$, making sure that $X_O\ne\es$ if $O\in\O'$. Then for $X:=\bigcup_{O\in\O} X_O$ a separation $\vs\in\tau$ will be represented by the partition~$(A,B)$ of~$X$ in which $B = \bigcup\{\,X_O\mid \vs\in O\,\}$ and $A = \bigcup\{\,X_O\mid \sv\in O\,\}$.%
   \COMMENT{}
   These ideas will be developed further in the next section.

\section{\boldmath Tree sets from \td s of graphs and matroids}\label{sec:Strees}

In this section we clarify the relationship between finite%
   \footnote{Tree-decompositions of infinite graphs and their tree sets will be treated in~\cite{ProfiniteTreeSets}.}
   \td s, the more general `$S$-trees' introduced in~\cite{TangleTreeAbstract}, and tree sets. Given a \td\ of a finite graph or matroid~$X$, the separations of~$X$ that correspond to the edges of the decomposition tree are always nested. If none of them is trivial or degenerate, i.e., if they form a tree set, then%
   \COMMENT{}
   the \td\ can essentially be recovered from this tree set.%
   \COMMENT{}
   The purpose of this section is to show how.\looseness=-1

The point of doing this is to establish that tree sets, which are more versatile for infinite combinatorial structures (even just for graphs) than \td s, are also no less powerful when they are finite: if desired, we can construct from any finite tree set of separations of a graph or matroid a \td\ whose tree edges correspond to precisely these separations.

Given a graph~$G$ and a family $\V = (V_t)_{t\in T}$ of subsets of its vertex set indexed by the node of a tree~$T$, the pair $(\V,T)$ is called a {\em\td\/} of $G$ if $G$ is the union of the subgraphs~$G[V_t]$ induced by these subsets, and $V_t\cap V_{t''}\sub V_{t'}$ whenever $t'$ lies on the $t$--$t''$ path in~$T$. The {\em adhesion sets\/} $V_{t_1}\cap V_{t_2}$ of~$(\V,T)$ corresponding to the edges $e=t_1 t_2$ of~$T$ then separate the sets $U_1 := \bigcup_{t\in T_1}\! V_t$ from $U_2 := \bigcup_{t\in T_2}\! V_t$ in~$G$, where $T_i$ is the component of $T-e$ containing~$t_i$, for $i=1,2$; see~\cite{DiestelBook16}. These separations $\{U_1,U_2\}$ are the separations of~$G$ {\em associated with\/} $(T,\V)$, and with the edges of~$T$.

Tree-decompositions can be described entirely in terms of $T$ and the oriented separations $\alpha(t_1,t_2) := (U_1,U_2)$ of $G$ associated with its edges. Indeed, we can recover its parts $V_t$ from these separations as the sets $V_t = \bigcap\{\,B\mid(A,B)\in\sigma_t\}$, where $\sigma_t$ is the star of the separations $\alpha(x,t)$ with $x$ adjacent to~$t$ in~$T$. Our aim in this section is to see under what assumptions the \td\ can be recovered not only from this nested set~$\tau$ of separations together with the information of how it relates to~$T$, but from the set~$\tau$ alone.

In an intermediate step, let us use both $T$ and the set of separations corresponding to its edges to view $(\V,T)$ in the following more general set-up from~\cite{TangleTreeAbstract}.
Let $\vS$ be a separation system, and let $\F\sub 2^{\vec S}$.

\begin{DEF} An \emph{${S}$-tree\/} is a pair $(T,\alpha)$ of a tree%
   \COMMENT{}
   $T$ and a map $\alpha\colon\vec E(T)\to \vS$ such that, for every edge $xy$ of~$T$, if $\alpha(x,y)=\vs$ then $\alpha(y,x)=\sv$. An $S$-tree $(T,\alpha)$ is {\em over $\F\sub 2^{\vec S}$} if, in addition, for every node $t$ of~$T$ we have $\alpha(\vec F_t)\in\F$.
\end{DEF}

\noindent
(Recall from Section~\ref{sec:Gtrees} that $\vec F_t$ is the oriented star at~$t$ in~$T$.)

We  say that the set $\alpha(\vec F_t)\sub\vS$ is {\em associated with\/} $t$ in $(T,\alpha)$.
The sets $\F$ we shall consider will all be \emph{standard} for~$\vS$, which means that they contain every co-trivial singleton $\{\rv\}$ in~$\vS$.

Since \td s can be recovered from the $S$-trees they induce, as pointed out earlier, our remaining task is to see which $S$-trees can be recovered just from the set $\alpha(\vec E(T))$ of their separations. As it turns out, this will be possible once we have trimmed a given $S$-tree down to its `essence', which is done in three steps.

An $S$-tree $(T,\alpha)$ is {\em redundant\/} if it has a node $t$ of~$T$ with distinct neighbours $t',t''$ such that $\alpha(t,t') = \alpha(t,t'')$; otherwise it is {\em irredundant\/}. Redundant $S$-trees can be pruned to irredundant ones over the same~$\F$, simply by deleting those `redundant' branches of the tree:

\begin{LEM}\label{prune}
For every finite%
   \COMMENT{}
   $S$-tree $(T,\alpha)$ over some $\F\sub 2^\vS$ there is an irredundant $S$-tree $(T',\alpha')$ over~$\F$ such that $T'\sub T$ and $\alpha' = \alpha\restricts\vec E(T')$.%
   \COMMENT{}
\end{LEM}

\begin{proof}
Let $t\in T$ have neighbours $t',t''$ witnessing the redundance of~$(T,\alpha)$. Deleting from~$T$ the component $C$ of $T-t$ that contains~$t''$ turns $(T,\alpha)$ into an $S$-tree in which $\vec F_t$ has changed but $\alpha(\vec F_t)$ has not, and neither has $\alpha(\vec F_x)$ for any other node $x\in T-C$. So this is still an $S$-tree over~$\F$. As $T$ is finite, we obtain the desired $S$-tree $(T',\alpha')$ by iterating this step.%
   \COMMENT{}
   \end{proof}

An important example of ${S}$-trees are irredundant ${S}$-trees \emph{over stars}: those over some $\F$ all of whose elements are stars of separations. Since stars contain no degenerate separations, the same holds for the image of~$\alpha$ in such $S$-trees~$(T,\alpha)$.

More importantly, in an ${S}$-tree $(T,\alpha)$ over stars the map~$\alpha$ preserves the natural partial ordering on $\vec E(T)$ defined at the start of Section~\ref{sec:Gtrees}:

\begin{LEM}\label{preservele}
Let $(T,\alpha)$ be an irredundant $S$-tree over stars. Let $\ve,\vf\in\vec E(T)$.
\begin{enumerate}[\rm (i)]\itemsep=0pt
   \item If $\ve \le \vf$ then $\alpha(\ve)\le\alpha(\vf)$. In particular, the image of~$\alpha$ in~$\vS$ is nested.%
   \COMMENT{}
    \item If $\alpha(\ve) < \alpha(\vf)$%
   \COMMENT{}
   then $\ve < \vf$, unless either $\alpha(\ve) = \alpha(\fv)$ is small, or $\alpha(\ve)$ or $\alpha(\fv)$ is trivial.%
   \COMMENT{}
   \end{enumerate}
\end{LEM}

\begin{proof}
   (i) Assume first that $e$ and~$f$ are adjacent; then $\ve,\fv\in\vec F_t$ for some $t\in T$. As $(T,\alpha)$ is irredundant we have $\alpha(\ve)\ne\alpha(\fv)$, and hence $\alpha(\ve)\le\alpha(\vf)$ since $\alpha(\vec F_t)$ is a star. By induction on the length of the $e$--$f$ path in $T$ this implies~(i) also for nonadjacent $e$ and~$f$.

(ii) Suppose $\ve\not<\vf$. Since $e$ and $f$ are nested, we then have
 $$\ve \ge \vf\text{ \ or \ }\ve\ge\fv\text{ \ or \ }\ve\le\fv.$$
 If $\ve\le\fv$, we have $\alpha(\ve)\le\alpha(\fv)$ by~(i), while $\alpha(\fv) < \alpha(\ev)$ by assumption%
   \COMMENT{}
   and~\eqref{invcomp} (and the fact that $\alpha$ commutes with inversion).%
   \COMMENT{}
   If even $\alpha(\ve) < \alpha(\fv)$, then $\alpha(\ve)$ is trivial.%
   \COMMENT{}
   Otherwise, $\alpha(\ve) = \alpha(\fv) < \alpha(\ev)$ is small.%
   \COMMENT{}

   Suppose next that $\ve\ge\fv$. Then $\alpha(\fv)\le \alpha(\ve)$ by~(i), while $\alpha(\fv) < \alpha(\ev)$ by assumption. If even $\alpha(\fv) < \alpha(\ve)$ then $\alpha(\fv)$ is trivial. Otherwise, $\alpha(\ve) = \alpha(\fv) < \alpha(\ev)$ is small.%
   \COMMENT{}

   Suppose finally that $\ve\ge\vf$. Then $\alpha(\ve) < \alpha(\vf)\le \alpha(\ve)$ by assumption and~(i), a contradiction.%
   \COMMENT{}
 \end{proof}

By Lemma~\ref{preservele}\,(i), the separations in an irredundant $S$-tree over stars are nested. For redundant $S$-trees this need not be so: if $\alpha(\ve)=\alpha(\vf)$ for ${\ve,\vf\in \vec F(t)}$, then separations $\alpha(\vedash)$ with $\vedash < \ve$ may cross separations $\alpha(\vfdash)$ with $\vfdash < \vf$. This is because we defined stars of separations as sets, not as multisets: for $\ve$ and $\vf$ as above we do not require that $\alpha(\ve)\le\alpha(\fv)$ when we ask that $\alpha(\vec F_t)$ be a star, since $\alpha(\ve)=\alpha(\vf)$ are not distinct elements of~$\alpha(\vec F_t)$.

Lemma~\ref{preservele}\,(ii) is best possible in that all the cases mentioned can occur independently.%
   \COMMENT{}
   We also need the inequalities to be strict, unless we assume that the $S$-tree is tight (see below).

\medbreak

Two edges of an irredundant $S$-tree over stars cannot have orientations pointing towards each other that map to the same separation, unless this is trivial:

\begin{LEM}\label{NewLemma}
Let $(T,\alpha)$ be an irredundant $S$-tree over a set $\F$ of stars. Let $e,f$ be distinct edges of~$T$ with orientations $\ve < \fv$ such that $\alpha(\ve) = \alpha(\vf) =:\vr$. Then $\vr$ is trivial.
\end{LEM}

\begin{proof}
If $\alpha$ maps all $\vedash$ with $\ve < \vedash < \fv$ to $\vr$ or to~$\rv$, then the $e$--$f$ path in $T$ has a node with two incoming edges mapped to~$\vr$. This contradicts our assumption that $(T,\alpha)$ is irredundant. Hence there exists such an edge~$\vedash$ with $\alpha(\vedash) = \vs$ for some $s\ne r$. Lemma~\ref{preservele} implies that $\vr = \alpha(\ve) \le \alpha(\vedash) \le\alpha(\fv) = \rv$, so $\vr\le\vs$ as well as $\vr \le \sv$ by~\eqref{invcomp}. As $s\ne r$ these inequalities are strict, so $s$ witnesses that $\vr$ is trivial.
\end{proof}

Let us call an $S$-tree $(T,\alpha)$  {\em tight\/} if its sets $\alpha(\vec F_t)$ are antisymmetric.%
   \COMMENT{}
   The name `tight' reflects the fact that from any $S$-tree we can obtain a tight one over the same~$\F$ by contracting edges:

\begin{LEM}\label{tight}
For every finite%
   \COMMENT{}
   $S$-tree $(T,\alpha)$ over some $\F\sub 2^\vS$ there exists an irredundant and tight $S$-tree $(T',\alpha')$ over~$\F$ such that $T'$ is a minor of $T$ and $\alpha' = \alpha\restricts\vec E(T')$.
\end{LEM}

\begin{proof}
By Lemma~\ref{prune} we may assume that $(T,\alpha)$ is irredundant. Consider any node $t$ of~$T$ for which $\alpha(\vec F_t)$ is not antisymmetric. Then $t$ has distinct neighbours $t',t''$ such that $\alpha(t',t) = \alpha(t,t'') =: \vs$.%
    \COMMENT{}
    Let $T'$ be obtained from~$T$ by contracting one of these edges and any branches of~$T$ attached to~$t$ by edges other than these two.%
   \footnote{In other words: delete the component of $T-t't - tt''$ containing~$t$, and join $t'$ to~$t''$. Then think of the edge $t't''$ as the old edge $t't$, so that $E(T')\sub E(T)$ as desired.}
   Let $\alpha'(t',t'') := \vs$ and $\alpha'(t'',t') := \sv$, and otherwise let $\alpha':= \alpha\!\restriction\! \vec E(T')$. Then $(T',\alpha')$ is again an $S$-tree, whose sets $\vec F_t$ in $T'$ are the same as they were in~$T$, for every $t'\in T'$. In particular, $(T',\alpha')$ is still irredundant and an $S$-tree over~$\F$. Iterate this step until the $S$-tree is tight.
   \end{proof}

Let us call an $S$-tree $(T,\alpha)$ {\em essential\/} if it is irredundant, tight, and $\alpha(\vec E(T))$ contains no trivial separation. Let the {\em essential core\/} of a set $\F\sub 2^\vS$ be the set 
of all $F'\sub\vS$ obtained from some $F\in\F$ by deleting all its trivial elements.%
   \COMMENT{}
   And call $\F$ {\em essential\/} if it equals its essential core.

An $S$-tree over stars can be made essential by first pruning it to make it irredundant (Lemma~\ref{prune}), then contracting the pruned tree to make it tight (Lemma~\ref{tight}), and finally deleting any edges mapping to trivial separations:

\begin{LEM}\label{essentialStrees}
For every irredundant and tight finite%
   \COMMENT{}
   $S$-tree $(T,\alpha)$ over a set~$\F$ of stars there is an essential $S$-tree $(T',\alpha')$ over the essential core of~$\F$ such that $T'\sub T$ and $\alpha' = \alpha\restricts\vec E(T')$.
\end{LEM}

\begin{proof}
Recall that if $\vs\in\vS$ is trivial then so is every $\vr\le\vs$. By Lemma~\ref{preservele}, therefore, the set~$\vec F$ of all edges $\ve\in\vec E(T)$ such that $\alpha(\ve)$ is trivial is closed down in~$\vec E(T)$. Hence the subgraph $T'$ of $T$ obtained by deleting each of these edges~$e$ together with the initial vertex of $\ve$ is connected, and therefore a tree: it may be edgeless, but it will not be empty, since the target vertex of any maximal edge in~$\vec F$ will be in~$T'$.%
   \COMMENT{}
   Clearly, $(T',\alpha')$ has all the properties claimed.
\end{proof}

Combining Lemmas \ref{tight} and~\ref{essentialStrees}, we obtain

\begin{COR}
For every finite $S$-tree $(T,\alpha)$ over a set~$\F$ of stars there is an essential $S$-tree $(T',\alpha')$ over the essential core of~$\F$%
   \COMMENT{}
   such that $T'$ is a minor of~$T$ and $\alpha' = \alpha\restricts\vec E(T')$.\qed
\end{COR}

\begin{LEM}\label{injective}
For every essential $S$-tree~$(T,\alpha)$ over stars the map $\alpha$~is injective.
\end{LEM}

\begin{proof}
Suppose there are distinct $\ve,\vf\in\vec E(T)$ with $\alpha(\ve)=\alpha(\vf) =:\vs$.
Then also $e\ne f$: otherwise $\ve=\fv$ making $\vs$ degenerate,%
   \COMMENT{}
   which cannot happen since $(T,\alpha)$ is over stars.%
   \COMMENT{}

Suppose first that $\ve < \vf$ in the natural order on~$\vec E(T)$. By Lemma~\ref{preservele}, every $\vedash\in\vec E(T)$ with $\ve\le\vedash\le\vf$ satisfies $\vs = \alpha(\ve)\le\alpha(\vedash) \le \alpha(\vf) = \vs$, so $\alpha(\vedash)=\vs$. As $\vs\ne\sv$,%
   \COMMENT{}
   this contradicts our assumption that $(T,\alpha)$ is tight, since $\vs,\sv\in\alpha(\vec F_t)$ for the terminal node $t$ of~$\ve$.%
   \COMMENT{}

Suppose now that $\ve < \fv$. Then $\vs$ is trivial by Lemma~\ref{NewLemma}, contradicting our assumption that $(T,\alpha)$ is essential.

Up to renaming $\ve$ and~$\vf$ as $\ev$ and~$\fv$, this covers all cases.
\end{proof}

As we have seen, a \td\ $(\V,T)$ of a graph or matroid can be recaptured from the structure of $T$ and the family $(\alpha(\ve)\mid\ve\in\vec E(T))$ of oriented separations it induces, i.e., from the $S$-tree $(T,\alpha)$. We can now show that if this $S$-tree is essential%
   \COMMENT{}
   (which we may often assume, cf.\ Lemmas \ref{prune}, \ref{tight} and~\ref{essentialStrees}), it can in turn be recovered from just the {\em set\/} of these oriented separations.

Recall that a subset $\sigma$ of a nested separation system $(\vS,\le\,,\!{}^*)$ {\em splits\/} it if $\sigma$ is the set of maximal elements of some consistent orientation of~$\vS$ and $\vS\sub\dcl(\sigma)$ (which is automatic when $\vS$ is finite). If $\vS$ has no degenerate elements, these sets~$\sigma$ are proper stars in~$\vS$, its {\em splitting stars\/}, and they contain no separations that are trivial in~$\vS$ \cite[Lemma~4.4]{AbstractSepSys} or co-trivial \cite[Lemma~3.5]{AbstractSepSys}. Except for one exceptional case where $\sigma$ contains a small separation, they are precisely the maximal proper stars in~$\vS$ \cite[Lemma~4.5]{AbstractSepSys}.

Let us say that $\vS$ is a nested separation system%
   \COMMENT{}
   {\em over\/} $\F\sub 2^{\vS}$ if all its splitting sets%
   \COMMENT{}
   lie in~$\F$.
   \COMMENT{}

\begin{THM}\label{finiteStrees}
Let $\vS$ be a finite separation system, and $\F\sub 2^\vS\!$ a set of stars.
\begin{enumerate}[\rm(i)]\itemsep0pt\vskip-\smallskipamount\vskip0pt
   \item If $\vS$ is nested and over~$\F$,%
   \COMMENT{}
   then there exists an essential $S$-tree $(T,\alpha)$ over~$\F$ such that $\alpha$ is an isomorphism of tree sets between the edge tree set of~$T$ and the regularization of the essential core of~$\vS$. In particular, the sets $\{\,\alpha(\vec F_t)\mid t\in T\,\}\in\F$ are precisely the sets%
   \COMMENT{}
   splitting~$\vS$.
   \item If $(T,\alpha)$ is an essential $S$-tree over~$\F$, then $\alpha$ is a tree set isomorphism between the edge tree set of~$T$ and a tree set over~$\F$ which is the regularization%
   \COMMENT{}
   of a tree set in~$\vS$.%
   \COMMENT{}
\end{enumerate}
\end{THM}

\begin{proof}
(i) Note first that $\vS$ has no degenerate element, because this would form a singleton set splitting~$\vS$ \cite[Lemma~4.4]{AbstractSepSys}, contradicting our assumption that $\vS$ is a separation system over stars.%
   \COMMENT{}
Let $\tau$ be the essential core of~$\vS$.%
   \COMMENT{}
   By \cite[Lemma~4.4]{AbstractSepSys}, the stars splitting~$\tau$ are precisely the sets that split~$\vS$; in particular, they lie in~$\F$.

Since $\tau$ is essential, it has a regularization~$\tau'$. This is a regular tree set, which has the same orientations, and hence the same splitting stars, as~$\tau$ (and~$\vS$).%
   \COMMENT{}
   Let $T = T(\tau')$ be the tree from Lemma~\ref{Ttau}. By Theorem~\ref{Trees}\,(i), the identity%
   \COMMENT{}
   is an isomorphism of tree sets between~$\tau'$ and the edge tree set of~$T$. In particular, the oriented stars~$\vec F_t$ at its nodes~$t$ are the splitting stars of~$\tau'$. Choosing as $\alpha$ the identity on~$\tau' = \vec E(T)$, we obtain $(T,\alpha)$ as desired.

(ii)%
   \COMMENT{}
   By Lemma~\ref{injective}, the map $\alpha$~is injective, and by Lemma~\ref{preservele}\,(i) it preserves the natural ordering of~$\vec E(T)$. By Lemma~\ref{preservele}\,(ii), also $\alpha^{-1}$ preserves the ordering of every {\em antisymmetric\/} subset of its domain.%
   \COMMENT{}

Since $(T,\alpha)$ is essential, this means that $\alpha(\vec E(T))$ is an essential tree set.%
   \COMMENT{}
   It thus has a regularization~$\tau'$, so that $\alpha\colon\vec E(T)\to\tau'$ is an order isomorphism. Since $\alpha$ also commutes with the involutions on the separation systems $\tau(T)$ and~$\vS$, this makes $\alpha$ into an isomorphism of tree sets between $\tau(T)$ and~$\tau'$.%
   \COMMENT{}
 \end{proof}

In a nutshell, Theorem~\ref{finiteStrees}\,(i) tells us that every%
   \COMMENT{}
   tree set $\vS$ over stars can be represented by an $S$-tree over the same stars, and Theorem~\ref{finiteStrees}\,(ii) tells us that this $S$-tree $(T,\alpha)$ from~(i) is unique in the following sense: for any other such $S$-tree $(T',\alpha')$, the composition $\alpha^{-1}\circ\alpha'$ maps the edge tree set of~$T'$ to that of~$T$ as an isomorphism of tree sets.%
   \footnote{Note that $\alpha$ and $\alpha'$ themselves will fail to be tree set isomorphisms if their image~$\vS$ is (essential but) irregular.}
   By Theorem~\ref{Trees}\,(ii), then, the trees $T'$ and~$T$ are also canonically isomorphic: by a graph isomorphism $V(T')\to V(T)$ that induces the map $\alpha^{-1}\circ\alpha'$ on their edges.%
   \COMMENT{}

The nested separation systems whose representations by $S$-trees we have described were all tree sets:%
   \COMMENT{}
   they have neither trivial nor degenerate elements. What about the others?

The nested separation systems $\vS$ that have a degenerate element~$\vs=\sv$ are easy to desribe directly. It is easy to see that such an $S$ has no nontrivial element other than~$s$ \cite[Section~3]{AbstractSepSys}; in particular, it has no other degenerate element. Hence $S$ has a unique consistent orientation~$O$, the set consisting of $\vs$ and all the trivial elements of~$\vS$. Since $s$~is the unique nontrivial witness to their triviality,%
   \COMMENT{}
   $\vs$~is the greatest element of~$O$. All the other $\vr,\vrdash\in O$ satisfy $\vr < \vs=\sv > \rvdash$, so $O$ is like a star~-- except that formally it is not, because stars must not contain degenerate elements.

It remains to consider the inessential finite nested separation systems without degenerate elements. These can also represented by $S$-trees. The only difference is that those $S$-trees will not be unique, even up to graph isomorphism.%
   \COMMENT{}
   But for every such $S$-tree $(T,\alpha)$ the edges of $T$ which $\alpha$ maps to nontrivial separations in $S$ form a (connected) subtree~$T'$, where $(T',\alpha')$ with $\alpha' = \alpha\restricts\vec E(T')$ is the essentially unique $S$-tree from Theorem~\ref{finiteStrees}\,(i).

\bibliographystyle{plain}
\bibliography{collective}

\small\vfill\noindent Version 23.2.2017

\end{document}